\documentclass{amsart}

\usepackage[top=1 in,bottom=1 in, left=1 in, right= 1 in, includehead,includefoot]{geometry}

          \usepackage{amssymb}
          \usepackage{amsmath}
          \usepackage{amsthm}
          \usepackage{enumerate}
          \usepackage{url}
          \usepackage{color}

\usepackage[numbers,sort&compress]{natbib}
\usepackage{epsfig}
\usepackage{ifthen}

\newcommand{\comment}[1]{}
\def\fddto{\xrightarrow{\textit{f.d.d.}}}
\newcommand{\ind}{{\bf 1}}

\def\inddd#1{{\ind}_{\left\{#1\right\}}}
\def\indn#1{\{#1_n\}_{n\in\N}}

\newcommand{\proba}{\mathbb P}
\newcommand{\esp}{{\mathbb E}}

\newcommand{\inv}{^{-1}}
\newcommand{\cov}{{\rm{Cov}}}

\newcommand{\eqnh}{\begin{eqnarray*}}
\newcommand{\eqne}{\end{eqnarray*}}
\newcommand{\eqnhn}{\begin{eqnarray}}
\newcommand{\eqnen}{\end{eqnarray}}
\newcommand{\equh}{\begin{equation}}
\newcommand{\eque}{\end{equation}}

\def\summ#1#2#3{\sum_{#1 = #2}^{#3}}

\def\sif#1#2{\sum_{#1=#2}^\infty}

\newcommand{\eqd}{\stackrel{d}{=}}

\def\topp#1{^{(#1)}}

\def\abs#1{\left|#1\right|}

\def\ccbb#1{\left\{#1\right\}}

\def\pp#1{\left(#1\right)}

\def\bb#1{\left[#1\right]}

\def\mmid{\;\middle\vert\;}

\def\floor#1{\left\lfloor #1 \right\rfloor}

\def\aa#1{\left\langle #1\right\rangle}
\def\saa#1{\langle#1\rangle}
\def\vv#1{{\boldsymbol #1}}

\def\vvt{{\boldsymbol t}}

\def\qmand{\quad\mbox{ and }\quad}

\def\qmwith{\quad\mbox{ with }\quad}
\def\mfa{\mbox{ for all }}

\def\wt#1{\widetilde{#1}}
\def\wb#1{\overline{#1}}
\def\what#1{\widehat{#1}}

\def\limn{\lim_{n\to\infty}}

\def\limsupn{\limsup_{n\to\infty}}
\def\liminfn{\liminf_{n\to\infty}}
\def\weakto{\Rightarrow}

\def\R{{\mathbb R}}

\def\N{{\mathbb N}}

\def\B{{\mathbb B}}

\def\calC{\mathcal C}

\def\calE{\mathcal E}

\def\calI{\mathcal I}

\def\calL{\mathcal L}
\def\calM{\mathcal M}

\def\calX{\mathcal X}

\def\topp#1{^{\scriptscriptstyle (#1)}}

\def\mab{\calM_{\alpha,\beta}}

\def\ddelta#1{\delta_{\pp{#1}}}
\usepackage{epsfig}
\usepackage{ifthen}

\newtheorem{Thm}{Theorem}[section]
\newtheorem{Lem}[Thm]{Lemma}
\newtheorem{Prop}[Thm]{Proposition}
\newtheorem{Coro}[Thm]{Corollary}

\theoremstyle{definition}
\newtheorem{Rem}[Thm]{Remark}

\numberwithin{equation}{section}

\usepackage{epsfig}
\usepackage{ifthen}

\title[An aggregated model]{An aggregated model for Karlin stable processes}

\author{Yi Shen}
\address{
Yi Shen,
Department of Statistics and Actuarial Science,
University of Waterloo,
Mathematics 3 Building,
200 University Avenue West
Waterloo, Ontario N2L 3G1,
Canada.
}
\email{yi.shen@uwaterloo.ca}

\author{Yizao Wang}
\address
{
Yizao Wang,
Department of Mathematical Sciences,
University of Cincinnati,
2815 Commons Way, ML--0025,
Cincinnati, OH, 45221-0025, USA.
}
\email{yizao.wang@uc.edu}
\author{Na Zhang}
\address
{
Na Zhang,
Department of Mathematics,
Towson University,
8000 York Road
Towson, MD 21252}
\email{nzhang@towson.edu}

\date{\today}

\begin{document}\sloppy
\begin{abstract}
An aggregated model is proposed, of which the partial-sum process scales to the Karlin stable processes recently investigated in the literature. The limit extremes of the proposed model, when having regularly-varying tails, are characterized by the convergence of the corresponding point processes. The proposed model is an extension of an aggregated model proposed by \citet{enriquez04simple} in order to approximate fractional Brownian motions with Hurst index $H\in(0,1/2)$, and is of a different nature of the other recently investigated Karlin models which are essentially based on infinite urn schemes.
\end{abstract}

\maketitle
\section{Introduction and main results}

\subsection{Karlin stable processes}
The Karlin stable processes are a family of self-similar symmetric $\alpha$-stable (S$\alpha$S) stochastic processes, $\alpha\in(0,2]$, with stationary increments that recently appeared in the literature \citep{durieu16infinite,durieu20infinite}. A Karlin S$\alpha$S process has a memory parameter $\beta\in(0,1)$. In the case $\alpha=2$, the process becomes a fractional Brownian motion with Hurst index $H = \beta/2\in(0,1/2)$. The Karlin stable processes exhibit long-range dependence \citep{samorodnitsky16stochastic,pipiras17long,beran13long}, and they first appeared as scaling limits of the so-called Karlin model, which is as an infinite urn scheme, of which the law on the urns has a power-law decay \citep{karlin67central,gnedin07notes}, with certain randomization.
The Karlin model and its recent variations  have attracted attentions in the literature of stochastic processes as they serve as simple models that exhibit long-range dependence. Notable variations and extensions include one to set-indexed models  \citep{fu21simulations} that include and extend the set-indexed fractional Brownian motions \citep{herbin06set},   and another recent one  to hierarchical models \citep{iksanov21functional}.

  We first recall the Karlin stable processes  $\{\zeta_{\alpha,\beta}(t)\}_{t\ge0}$, 
  and explain
   how it arises from the Karlin model 
   with randomization as in \citep{durieu16infinite,durieu20infinite}.
 Throughout, we assume
 $\alpha\in (0,2]$ and $\beta\in (0,1)$. Then, $\zeta_{\alpha,\beta}$ is a symmetric $\alpha$-stable (S$\alpha$S) process, of which the characteristic function of  finite-dimensional distributions is, for any $d\in\N\equiv \ccbb{1, 2, \cdots}$, $t_1,\cdots, t_d\ge 0$, $\theta_1, \cdots, \theta_d\in\R$,
\equh\label{eq:zeta fdd}
\esp\exp\pp{i\sum_{j=1}^{d}\theta_j  \zeta_{\alpha,\beta}(t_j)}
=\exp\pp{-\frac{\beta}{\Gamma(1-\beta)\mathsf C_\alpha}\int_{0}^{\infty}\esp \abs{\sum_{j=1}^{d}\theta_j \inddd{N(t_jq) \text { odd}}}^{\alpha}q^{-\beta-1}dq
},
\eque
where on the right-hand side, $N$ is a standard Poisson process (the probability spaces involved on both sides are not necessarily the same), and
 \[
 \mathsf C_\alpha = \begin{cases}
  \pp{\int_0^\infty x^{-\alpha}\sin xdx}\inv, & \mbox{ if } \alpha\in(0,2),\\
  2 & \mbox{ if } \alpha = 2.
  \end{cases}
 \]
It follows from the representation above that the process is self-similar with index $\beta/\alpha$ and with stationary increments \citep{durieu20infinite}. Moreover, when $\alpha=2$ it is a fractional Brownian motion with Hurst index $H = \beta/2$ up to a multiplicative constant (see \eqref{eq:fBm} below).

When $\alpha\in(0,2)$, \eqref{eq:zeta fdd} has a corresponding series representation as a well-known fact on stable processes \citep{samorodnitsky94stable}. However, for our discussions later we shall need to work with another series representation of the process $\zeta_{\alpha,\beta}$ restricted to $t\in[0,1]$ as follows. In this case, introduce first
 \equh\label{eq:xi}
 \xi:=\sif\ell1\summ j1{Q_{\beta,\ell}}\ddelta{\varepsilon_\ell\Gamma_\ell^{-1/\alpha},U_{\ell,j}},
 \eque
 where $\{\Gamma_\ell\}_{\ell\in\N}$ is the collection of consecutive arrival times of a standard Poisson process on $\R_+$, $\{\varepsilon_\ell\}_{\ell\in\N}$ are i.i.d.~Radamacher random variables, $\{Q_{\beta,\ell}\}_{\ell\in\N}$ are i.i.d.~copies of a $\beta$-Sibuya random variable (see \eqref{eq:Sibuya} below), and $\{U_{\ell,j}\}_{\ell,j\in\N}$ are i.i.d.~uniform random variables on $(0,1)$. All four families of random variables are assumed to be independent. Then, we also have the following series representation of the Karlin stable processes, {\em restricted to $t\in[0,1]$},
 \equh\label{eq:zeta series}
 \ccbb{\zeta_{\alpha,\beta}(t)}_{t\in[0,1]} \eqd \ccbb{\sif\ell1 \frac{\varepsilon_\ell}{\Gamma_\ell^{1/\alpha}}\inddd{\summ j1{Q_{\beta,\ell}}\inddd{U_{\ell,j}\le t} \ \rm odd}}_{t\in[0,1]}, \alpha\in(0,2),\beta\in(0,1),
 \eque
The fact that the representations  \eqref{eq:zeta fdd} and \eqref{eq:zeta series}  are equivalent is
 recalled in Lemma \ref{lem:1} (following a more general result in \citep[Theorem 2.1]{fu21simulations}).

Now we explain the so-called {\em randomized Karlin model} in \citep{durieu16infinite,durieu20infinite}, for comparison purpose only (see Remark \ref{rem:comparison}).  Let $\indn Y$ be i.i.d.~$\N$-valued random variables with $\proba(Y_1 = k) \sim k^{-1/\beta}$ as $k\to\infty$ for some $\beta\in(0,1)$ (we only present a simple version; a slowly varying function is allowed in general). Let $\indn X$ be i.i.d.~random variables independent from $\indn Y$, and assume in addition that $X_1$ is symmetric with $1-\esp \exp(i\theta X_1)\sim \sigma_X^\alpha |\theta|^\alpha$ as $\theta\to 0$. Consider
 the partial-sum process
\equh\label{eq:Sn0}
S_n := \summ j1n (-1)^{K_{j
,Y_j}}
X_{Y_j} =  \sif\ell1 X_\ell\inddd{K_{n,\ell}\ \rm odd} \qmwith K_{n,\ell} := \summ j1n \inddd{Y_j = \ell}, \quad n,\ell\in\N.
\eque
Then, one can show that for some explicit constant $C_{\alpha,\beta}$,
\[
\frac1{n^{\beta/\alpha}}\ccbb{S_{\floor{nt}}}_{t\in[0,1]}\fddto C_{\alpha,\beta}\ccbb{\zeta_{\alpha,\beta}(t)}_{t\in[0,1]}.
\]
Note that when $\indn X$ are i.i.d.~Radamacher random variables, in view of the first expression 
in \eqref{eq:Sn0}
above $\{S_n\}_{n\in\N}$ can be interpreted as a correlated random walk with $\pm1$ steps that scales to a fractional Brownian motion with Hurst index $H = \beta/2$ \citep{durieu16infinite}.

\subsection{An aggregated model}
We propose a one-dimensional aggregated model as follows. The model actually extends a previous one by \citet{enriquez04simple} (see Remark \ref{rem:Enriquez}).  Let $q$ be a random 
 parameter taking values from $(0,1)$, and
 given $q$, let $\{\eta\topp q_j\}_{j\in\N}$ be a sequence of conditionally i.i.d.~Bernoulli random variables with parameter $q$.
 Let $\calX$ be a symmetric random variable, independent from $q$ and $\{\eta\topp q_j\}_{j\in\N}$.
Let $\alpha'>0$ be another parameter.  Then we introduce
\begin{equation}\label{eq:X_j}
X_j:=\frac{\calX}{q^{1/\alpha'}}\cdot (-1)^{\tau_j\topp q}\eta_j\topp q \qmwith \tau_j\topp q:= \summ k1j \eta\topp q_k, \quad j\in\N.
\end{equation}
In words, $X_j = 0$ whenever $\eta_j\topp q = 0$, and for those $j\in\N$ such that $\eta_j\topp q=1$, $X_j$ takes the same value $\calX/q^{1/\alpha'}$, but with alternating signs. One can check that $\{X_n\}_{n\in\N}$ forms a stationary sequence of random variables. The partial-sum process is then
\equh\label{eq:S_n}
S_n:=\summ j1n X_j = \frac \calX{q^{1/\alpha'}}\inddd{\tau_n\topp q \ \rm odd}, n\in\N.
\eque
Note that there is no summation involved in the second expression above, and $S_n \ne 0$ implies  necessarily that $\tau\topp q_n$ is odd.  The simple expression is essentially due to the alternating signs.
Next, introduce
\equh\label{eq:notations}
\pp{(\calX_i, q_i, \{\eta_{i,j}\topp{q_i}\}_{j\in\N},\{\tau_{i,j}\topp{q_i}\}_{j\in\N})}_{i\in\N}\stackrel{i.i.d.}\sim \pp{\calX,q, \{\eta\topp q_j\}_{j\in\N},\{\tau\topp{q}_j\}_{j\in\N}},
\eque
and for each copy let $\{S_n\topp i\}_{n\in\N}$ denote the corresponding partial-sum process.
We are interested in the  aggregated model, for an increasing sequence of positive integers $\{m_n\}_{n\in\N}$,
\[
\what \calX_{n,j} := \summ i1{m_n}\frac{\calX_i}{q_i^{1/\alpha'}}\cdot(-1)^{\tau_{i,j}\topp{q_i}}\eta_{i,j}\topp{q_i}, \quad n\in\N, j=1,\dots,n,
\]
and its corresponding partial-sum process
\[
\ccbb{\what S_n(t)}_{t\in[0,1]} :=\ccbb{\summ j1{\floor{nt}}\what \calX_{n,j}}_{t\in[0,1]} \equiv \ccbb{\sum_{i=1}^{m_n} S_{\floor{nt}}\topp i}_{t\in[0,1]} \equiv \ccbb{\summ i1{m_n}\frac{\calX_i}{q_i^{1/\alpha'}}\inddd{\tau_{i,{\floor{nt}}}\topp {q_i} \ \rm odd}}_{t\in[0,1]}.
\]
Above, we provide three equivalent representations to better understand the process. We shall mostly use the third one in our analysis.

Now we specify the assumptions on $q$ and $\calX$. The random parameter $q$ is assumed to have the probability density function
 \begin{equation}\label{eq:q}
p(x) =  x^{-\rho}L(1/x), \quad x\in(0,1), \mbox{ for some } \rho<1,
 \end{equation}
where $L$ is a slowly varying function at infinity.
The random variable $\calX$ is assumed to be symmetric, and either to have finite second moment, or
\begin{equation}\label{eq:varepsilon_RV}
\wb F_{|\calX|}(x) \equiv   \proba\pp{|\calX|>x}\sim C_\calX x^{-\alpha},
   x>0, 
    \mbox{ for some $\alpha>0$ and $C_\calX>0$}.
\end{equation}
\subsection{Main results}
Our first result is a 
multivariate
central limit theorem.
\begin{Thm}\label{thm:1}
Assume \eqref{eq:q} holds. Assume the symmetric random variable $\calX$ satisfies one of the following two conditions:
\begin{enumerate}[(i)]
\item $\esp \calX^2<\infty$, and in this case set $\alpha=2$, $C_\calX  :=\esp\calX^2$.
\item \eqref{eq:varepsilon_RV} holds with $\alpha\in(0,2)$.  \end{enumerate} Further 
assume
\equh\label{eq:beta}
\beta := \gamma-1+\rho\in(0,1) \qmwith \gamma := \frac\alpha{\alpha'}.
\eque
Then, with $m_n$ satisfying
\begin{equation}\label{eq:m_n}
    \limn \frac{m_nL(n)}{n^{1-\rho}}=\infty,
\end{equation}
and
\[
a_n = \pp{C_\calX \frac{\Gamma(1-\beta)}\beta \cdot n^\beta m_nL(n)}^{1/\alpha},
\]
we have
\[
\left\{\frac{\what S_{n}(t)}{a_n}\right\}_{t\in[0,1]}\fddto \{\zeta_{\alpha,\beta}(t)\}_{t\in[0,1]}.
\]
\end{Thm}

Regarding scaling limits of extremes, our second result is a convergence of point processes.
\begin{Thm}\label{thm:PPP}
Assume \eqref{eq:q}, \eqref{eq:varepsilon_RV} with $\alpha>0$ and \eqref{eq:beta}. Assume that, in addition to \eqref{eq:m_n}, $m_n\le C n^\kappa$ for some $\kappa\in(0,2\beta/(\alpha-2))$ if $\alpha\ge 2$ (so $\alpha=2$ means
 that $m_n$ grows at a polynomial rate).  We have
\[
\xi_n:=\summ j1n\ddelta{\summ i1{m_n}\calX_i\eta\topp{q_i}_{i,j}/(a_nq_i^{1/\alpha'}),j/n} \weakto \xi,
\]
in $\mathfrak M_p((\wb \R\setminus\{0\})\times [0,1])$, where $\xi$ is as in \eqref{eq:xi}.
\end{Thm}
Above and below, $\mathfrak M_p(E)$ is the space of Radon point measures on the metric space $E$, equipped with vague topology. Our reference for point processes and convergence is \citep{resnick87extreme}.
In the case $\alpha<2$, Theorem \ref{thm:PPP} contains more information regarding the limit of the partial-sum process, 
and provides a second proof for Theorem \ref{thm:1},
 as discussed in Section \ref{sec:second proof}. Moreover, Theorem \ref{thm:PPP} also implies extremal limit theorems regarding the proposed model, as explained in 
Section \ref{sec:RSM}.
 In particular, the choice of $a_n$ is such that
\[
\proba\pp{\frac{|\calX|}{q^{1/\alpha'}}>a_n x, \tau_n\topp q\ne 0}\sim \frac{x^{-\alpha}}{m_n}, \mfa x>0.
\]
As the key of Theorem \ref{thm:PPP}, a more refined conditional limit theorem given the event above is in Proposition \ref{prop:1}.

We conclude the introduction with a few remarks.
\begin{Rem}\label{rem:tightness}
For the central limit theorem, we only prove the convergence of finite-dimensional distributions to the Karlin stable process for
 $\alpha\in(0,2]$, without the tightness. The tightness is a challenging issue and actually, in \citep{durieu20infinite}, the tightness for the randomized Karlin model was only proved for $\alpha\in(0,1)$, and the tightness remains an open question for $\alpha\in[1,2)$ (for the Gaussian case the tightness was proved in \citep{durieu16infinite}). It is also an open question to show that the Karlin stable process has a version in $D$ space for $\alpha\in[1,2)$.
\end{Rem}

\begin{Rem}\label{rem:Enriquez}
The main inspiration of this paper came from a paper of \citet{enriquez04simple}, and our model is in fact a generalization of a model proposed therein. The goal of \citep{enriquez04simple} was to provide an approximation of fractional Brownian motion with Hurst index $H\in(0,1)$ by aggregation of independent correlated random walks. Two models were proposed therein and the second was for $H\in(0,1/2)$, 
recalled here.
Consider again random variable $q$ with probability density function
\[
(1-2H)2^{1-2H}q^{-2H}\ind_{\{q\in(0,1/2)\}}.
\]
Then a sequence of random variables $\{\varepsilon_n\}_{n\in\N}$ is constructed as follows: $\varepsilon_1$ is a $\pm 1$-valued symmetric random variable and for each $n\geq 1$, the law of $\varepsilon_n$ is determined by
\begin{align*}
q & = \proba(\varepsilon_{2n}=\varepsilon_{2n-1}\mid \varepsilon_1, \cdots,\varepsilon_{2n-1}, q)  = 1-\proba(\varepsilon_{2n}=-\varepsilon_{2n-1}\mid \varepsilon_1, \cdots,\varepsilon_{2n-1}, q), n\in\N,
\end{align*}
and
$\varepsilon_{2n+1}=-\varepsilon_{2n}, n\in\N$. Then, consider $\calX_n :=(\varepsilon_{2n-1}+\varepsilon_{2n})/(2\sqrt q), n\in\N$.
In this way, their model fits into our setup with $\alpha=\alpha' = 2, \rho = 2H$ (see \citep[p. 209]{enriquez04simple} for details), and our Theorem \ref{thm:1} includes \citep[Corollary 3]{enriquez04simple} as a special case for $\alpha=2$ (but without tightness). 
\end{Rem}

\begin{Rem}There is non-trivial dependence between the magnitude $\calX_i/q_i^{1/\alpha'}$ and the locations $\{j=1,\dots,n: \eta_{i,j}\topp{q_i} = 1\}$, via $q$, in the aggregated model. However, the dependence disappears in the limit. It is also remarkable that while our model has three parameters $\rho,\alpha,\alpha'$, the limit Karlin stable process has only two: $\alpha\in(0,2)$ and $\beta=\rho+\alpha/\alpha'-1\in(0,1)$.

Both observations can be explained 
 by the following representation of $\zeta_{\alpha,\beta}$ (compare also \eqref{eq:suggestion} in the proof of Theorem \ref{thm:1} later): essentially, the factor $q^{-1/\alpha'}$ in \eqref{eq:X_j} introduces an effect of change of measures in the limit.  Recall the characteristic function of $\zeta_{\alpha,\beta}$ in \eqref{eq:zeta fdd}, and write
\[
\int_0^\infty \esp\abs{\summ j1d\theta_j\inddd{N(t_j q) \ \rm odd}}^\alpha q^{-\beta-1}dq = \int_0^\infty \esp\abs{\summ j1d\theta_j\frac1{q^{\gamma/\alpha}}\inddd{N(t_j q) \ \rm odd}}^\alpha q^{-\rho}dq.
\]
Equivalently, for $\alpha<2$ we have another series representation as follows
\[
\ccbb{\zeta_{\alpha,\beta}(t)}_{t\ge 0} \eqd\ccbb{\sif\ell1 \frac{\varepsilon_\ell}{\Gamma_\ell^{1/\alpha}}q_\ell^{-\gamma/\alpha}\inddd{N_\ell(tq_\ell)\ \rm odd}}_{t\ge 0},
\]
where $\Gamma_\ell$ and $q_\ell$ are such that $\sif\ell1 \ddelta{\Gamma_\ell,q_\ell}$ is 
 a Poisson point process on $\R_+\times\R_+$ with intensity measure $dx (\beta/\Gamma(1-\beta))q^{-\rho}dq$, independent from the Rademacher random variables $\{\varepsilon_\ell\}_{\ell\in\N}$.
So in the limit process $\zeta_{\alpha,\beta}$,  $q^{-\gamma/\alpha}$ above eventually comes from the normalization $q^{-1/\alpha'}$ in \eqref{eq:X_j} and $q^{-\rho}$ comes from the density of $q$ (both after $1/n$-scaling as can be read from the proof later).
\end{Rem}

\begin{Rem}\label{rem:comparison}
In the randomized Karlin models \citep{durieu16infinite,durieu20infinite}, there are two sources of dependence. First, in \eqref{eq:Sn0}, the dependence is determined by the law of certain counting numbers being odd. For $\{S_n\}_{n\in\N}$ in \eqref{eq:Sn0}, with all $X_\ell=1$ it is known as an {\em odd-occupancy} process, counting by time $n$ how many urns have been sampled by an odd number of times. This process has been already investigated by \citet{karlin67central}, and the motivation of such a consideration dates
 back to \citet{spitzer64principles}. So with i.i.d.~$\{X_\ell\}_{\ell\in\N}$, $\{S_n\}_{n\in\N}$ becomes a randomized odd-occupancy process. The law of the occupancy numbers being odd, eventually, plays a crucial role in the underlying dependence structure of the limit Karlin stable process in \eqref{eq:zeta fdd} and \eqref{eq:zeta series}. Second, the original Karlin model also has a strong combinatorial flavor, as the sampling $\{Y_n\}_{n\in\N}$ induces a random partition of $\N$, essentially related to the Pitman--Yor partition with parameters $(\beta,0)$, to which the $\beta$-Sibuya distribution is intrinsically related \citep{pitman06combinatorial}.
There have been recent interests regarding various counting statistics for other combinatorial models, and often they lead to new stochastic processes of their own interest. For an example with a similar flavor, see \citep{alsmeyer17functional}.

In a sense, our proposed aggregated model and the limit theorems indicate that the counting of odd-occupancy numbers is much more fundamental than the underlying random partitions for the randomized Karlin models: our proposed model has a much less combinatorial flavor than the Karlin models, and yet they lead to the same scaling limits. On the other hand, it is well known that aggregated models with random coefficients may lead to stochastic processes with abnormal asymptotic behaviors (e.g.~\citep{kaj08convergence,mikosch07scaling}), 
and our result here provides another such example.
\end{Rem}
{The paper is organized as follows.} Section \ref{sec:KSP}
collects a few facts about Karlin stable processes. In Section \ref{sec:CLT} we prove Theorem \ref{thm:1}. In Section \ref{sec:PP} we prove Theorem \ref{thm:PPP} and explain its further connection to the so-called Karlin random sup-measures.
\section{Representations for Karlin stable processes}\label{sec:KSP}
We collect some facts on the Karlin stable processes that can be derived from the general Karlin stable set-indexed processes \citep{fu21simulations}. Let $Q_\beta$ denote the $\beta$-Sibuya distribution \citep{sibuya79generalized}, so that
\equh\label{eq:Sibuya}
\proba(Q_\beta = \ell) = \frac\beta{\Gamma(1-\beta)}\frac{\Gamma(\ell-\beta)}{\Gamma(\ell+1)}, \ell\in\N.
\eque
Let $\calC_\ell := \bigcup_{i=1}^{\ell}\{U_i\}$ denote the union of $\ell$ i.i.d.~random variables $\{U_i\}_{i\in\N}$ that are uniformly distributed over $(0,1)$. If $\ell$ is a random variable, then assume in addition that $\{U_i\}_{i\in\N}$ are independent from $\ell$.
\begin{Lem}\label{lem:1}
For $\alpha>0$ and $\beta\in(0,1)$,
\equh\label{eq:Poisson-Sibuya}
\int_0^\infty \esp_q \abs{\summ j1d\theta_j \inddd{N(t_jq)\ \rm odd}}^\alpha \frac\beta{\Gamma(1-\beta)}q^{-\beta-1}dq = \esp \abs{\summ j1d \theta_j\inddd{|\calC_{Q_\beta}\cap[0,t_j]|\ \rm odd}}^\alpha.
\eque
Therefore, \eqref{eq:zeta series} holds with $\alpha\in(0,2)$.
\end{Lem}
\begin{Rem}
Throughout, with a little abuse of notations,  when writing $\esp_q(\cdots)$ or $\proba_q(\cdots)$, we mean that the $q$ appears in $(\cdots)$ is viewed as a fixed constant instead of a random variable (e.g., $\esp_q(\cdots)$ on the left-hand side of \eqref{eq:Poisson-Sibuya} is viewed then as a function of $q$).
\end{Rem}
\begin{proof}[Proof of Lemma \ref{lem:1}]
First, let $\{N\topp q(t)\}_{t\ge 0}$ denote a Poisson process on $\R_+$ with constant intensity density $q$ on $\R_+$. Then, for every $q>0$ fixed,
\begin{align*}
\esp_q \abs{\summ j1d\theta_j \inddd{N(t_jq)\ \rm odd}}^\alpha & = \esp_q \abs{\summ j1d\theta_j \inddd{N\topp q(t_j)\ \rm odd}}^\alpha\\
& = \sif\ell1\esp_q \pp{\abs{\summ j1d\theta_j \inddd{N\topp q(t_j)\ \rm odd}}^\alpha\mmid N\topp q(1) = \ell}\proba_q(N\topp q(1) = \ell)\\
& = \sif\ell1\esp_q \abs{\summ j1d\theta_j \inddd{\abs{\calC_\ell\cap[0,t_j]}\ \rm odd}}^\alpha\proba_q(N\topp q(1) = \ell).
\end{align*}
Then, the left-hand side of \eqref{eq:Poisson-Sibuya} becomes, by Fubini's theorem,
\[
\sif\ell1\esp \abs{\summ j1d\theta_j \inddd{\abs{\calC_\ell\cap[0,t_j]}\ \rm odd}}^\alpha \cdot \frac\beta{\Gamma(1-\beta)}\int_0^\infty \proba_q(N\topp q(1) = \ell)q^{-\beta-1}dq.
\]
It remains to notice that the second factor above is simply
\[
\frac\beta{\Gamma(1-\beta)}\int_0^\infty \proba_q(N\topp q(1) = \ell)q^{-\beta-1}dq  = \frac\beta{\Gamma(1-\beta)}\int_0^\infty \frac{q^\ell}{\Gamma(\ell+1)}e^{-q}q^{-\beta-1}dq = \proba(Q_\beta = \ell).
\]
The desired identity \eqref{eq:Poisson-Sibuya} now follows. 
 \eqref{eq:zeta series} then follows by the well-known equivalence between stochastic-integral and series representations of S$\alpha$S processes \citep[Theorem 3.10.1]{samorodnitsky94stable}.
\end{proof}

The Karlin stable process $\zeta_{\alpha,\beta}$ has the following stochastic-integral representation \citep{samorodnitsky94stable}
\begin{equation}\label{eq:KSP}
\ccbb{\zeta_{\alpha,\beta}(t)}_{t\ge 0}\eqd\ccbb{\int_{\R_{+}\times\Omega'} \inddd{N'(tq)(\omega') \text { odd}} M_{\alpha}(dq, dw')}_{t\ge 0},
\end{equation}
where
$M_{\alpha}$ is a $\rm{S\alpha S}$ random measure on $\R_{+}\times \Omega'$, for another probability space $(\Omega',\proba')$ different from the one where the stochastic integral is defined on, with control measure $dm=(\beta/(\Gamma(1-\beta)\mathsf C_\alpha))q^{-\beta-1}dq\proba'\pp{d\omega'}$.
This is a standard so-called {\em doubly stochastic representation}, where the random measure $M_\alpha$ is defined on the by-default probability space $(\Omega,\proba)$, and $(\Omega',\proba')$ is a different space.
This representation is notationally convenient but not needed in our proofs. We refer the readers to \citep{samorodnitsky94stable} for more details regarding stochastic-integral representation for stable processes.

Note that the characteristic function \eqref{eq:zeta fdd} and the stochastic-integral representation \eqref{eq:KSP} both allow $\alpha=2$, and in this case the Karlin stable process becomes a fractional Brownian motion with Hurst index $\beta/2$ up to a multiplicative constant. A quick derivation is as follows, using \eqref{eq:KSP} and stochastic integrals with respect to Gaussian random measures: for $0<s<t$,
\begin{align*}
\cov\pp{\zeta_{2,\beta}(s),\zeta_{2,\beta}(t)} &=
\frac \beta{\mathsf C_\alpha\Gamma(1-\beta)}\int_0^\infty\proba_q\pp{N(sq)\ {\rm odd}, N(tq)\ {\rm odd}}q^{-\beta-1}dq\\
& = \frac\beta{\mathsf C_\alpha\Gamma(1-\beta)}\int_0^\infty \frac12(1-e^{-2qs})\frac 12(1+e^{-2q(t-s)})q^{-\beta-1}dq\\
& = \frac14\frac\beta{\mathsf C_\alpha\Gamma(1-\beta)}\int_0^\infty (1-e^{-2qs}-e^{-2qt}+e^{-2q(t-s)})q^{-\beta-1}dq.
\end{align*}
(Recall that for a Poisson random variable $N$ with $\lambda = \esp N$,  $\proba\pp{N \text{ is odd}}=\pp{1-e^{-2\lambda}}/2$.)
Using $\int_0^\infty (1-e^{-rq})q^{-\beta-1}dq = r^\beta \Gamma(1-\beta)/\beta$, we have
\equh\label{eq:fBm}
\cov\pp{\zeta_{2,\beta}(s),\zeta_{2,\beta}(t)} = 
2^{\beta-1}\mathsf C_\alpha^{-1}
\cdot \frac12\pp{s^\beta+t^\beta-|t-s|^\beta}, s,t\ge0.
\eque
So $\{\zeta_{2,\beta}(t)\}_{t\ge 0} \eqd 
2^{(\beta-1)/2}\mathsf C_\alpha^{-1/2}
\ccbb{\B^{\beta/2}(t)}_{t\ge 0}$, where $\{\B^{\beta/2}(t)\}_{t\ge 0}$ is a fractional Brownian motion with Hurst index $\beta/2$.

\section{Proof of Theorem \ref{thm:1}}\label{sec:CLT}
The proof is by computing the asymptotic characteristic function.
Consider the characteristic function $\phi_\calX(\theta) := \esp\exp(i\theta \calX)$. It is known that both two assumptions on $\calX$ in Theorem \ref{thm:1} can be unified into the following condition
\[
1-\phi_\calX(\theta)\sim \sigma_\calX^{\alpha}\abs{\theta}^\alpha \text{ as } \theta\to 0,
\]
where $\sigma_\calX^2:=\esp \calX^2/2<\infty$ and
$\sigma_\calX^\alpha:=C_\calX/\mathsf C_\alpha$ when $\alpha\in(0,2)$ 
(see \citep[Theorem 8.1.10]{bingham87regular} for the second).

Recall the characteristic function of the Karlin stable process in \eqref{eq:zeta fdd}. We shall rewrite it in a more convenient expression for our proof.
Throughout, for $d\in\N$,  write
\[
\Lambda_d:=\{0,1\}^d\setminus \{(0,\cdots,0)\},
\]
and for $\vv{\theta} = (\theta_1,\dots,\theta_d)\in\R^d, \vv\delta = (\delta_1,\dots,\delta_d)\in\Lambda_d$,
\[
\aa{\vv{\theta},\vv{\delta}}:=\sum_{j=1}^{d}\theta_j\delta_j.
\]
Recall that $N(t)$ is a standard Poisson process on $\R_{+}$, and write, with $\vv\delta\in\Lambda_d$ and $\vvt = (t_1,\dots,t_d)\in [0,1]^d$,
\[
 \ccbb{
\vv N(q\vvt)=\vv{\delta}\text{ mod } 2
 }  \equiv
 \ccbb{
 N(qt_j)=\delta_j  \text{ mod } 2 \text{ for all } j=1,\cdots, d
 }.
\]
Observe
\[
\esp_q\abs{\sum_{j=1}^{d}\theta_j \inddd{N(t_jq)\ \rm { odd}}}^{\alpha} = \sum_{\vv\delta\in\Lambda_d}|\saa{\vv\theta,\vv\delta}|^\alpha \proba_q(\vv N(q\vv t) = \vv\delta\ \rm mod~2).
\]
Therefore,
with
\[
\mathfrak m_{\alpha,\beta}(\vvt,\vv\delta):=\frac\beta{\Gamma(1-\beta)\mathsf C_\alpha}\int_{0}^{\infty}\proba_q\pp{\vv{N}(q\vvt)= \vv{\delta}\text{ mod } 2}q^{-\beta-1} dq,
\]
we see that \eqref{eq:zeta fdd} becomes
\[
\esp\exp\pp{i\sum_{j=1}^{d}\theta_j  \zeta_{\alpha,\beta}(t_j)} = \exp\pp{-\sum_{\vv\delta\in\Lambda_d}|\saa{\vv\theta,\vv\delta}|^\alpha  \mathfrak m_{\alpha,\beta}(\vvt,\vv\delta)
}.
\]

Now for the aggregated model, we start by computing the characteristic function of the finite-dimensional distributions of $S_{\floor{nt}}/a_n$ (recall $S_n$ in \eqref{eq:S_n}).
 Write
 \[
 A_{n,q}:=\ccbb{\tau_n\topp q \text { is odd}} \qmand
A^{\delta}:=\begin{cases}
A,& \mbox{ if } \delta=1,\\
A^{c},& \mbox{ if }\delta=0.
\end{cases}
\]
For any $d\in\N$, $\vv{t}\in [0,1]^d$,
write $n_j:=\floor{nt_j}, j=1, \cdots,d$, and
\[
 A_{n,q,\vvt,\vv\delta} := \bigcap_{i=1}^dA_{n_j,q}^{\delta_j}.
\]
Then, for $\vv{\theta}\in\R^d$,
 \begin{align}
  \esp\exp\pp{i\sum_{j=1}^{d}\theta_j\frac{S_{\floor{nt_j}}}{a_n}}
  &=  \esp\exp\pp{i\sum_{j=1}^{d}\frac{\theta_j\calX}{a_nq^{1/\alpha'}}\ind_{A_{n_j,q}}}\nonumber =  \esp\exp\pp{i\sum_{\vv{\delta}\in\Lambda_d}\frac{\saa{\vv{\theta},\vv{\delta}} \calX}{a_nq^{1/\alpha'}}\ind_{A_{n,q,\vvt,\vv\delta}}} \nonumber\\
& =  \esp\esp \pp{\exp\pp{i\sum_{\vv{\delta}\in\Lambda_d}\frac{\saa{\vv{\theta},\vv{\delta}} \calX}{a_nq^{1/\alpha'}}\ind_{A_{n,q,\vvt,\vv\delta}}}\mmid\calX,q}.  \label{eq:fdd}
 \end{align}
In the second step we used the fact that $A_{n,q,\vvt,\vv\delta}$ are disjoint for different $\vv\delta\in\Lambda_d$. Using the disjointness again, the inner conditional expectation of \eqref{eq:fdd} becomes
\begin{align*}
 \esp\pp{\prod_{\vv\delta\in\Lambda_d}\exp\pp{\frac{i\aa{\vv\theta,\vv\delta}\calX}{a_nq^{1/\alpha'}}\ind_{A_{n,q,\vvt,\vv\delta}}}\mmid\calX,q}& = \esp\pp{\sum_{\vv\delta\in\Lambda_d}
\exp\pp{\frac{i\aa{\vv\theta,\vv\delta}\calX}{a_nq^{1/\alpha'}}}\ind_{A_{n,q,\vvt,\vv\delta}}+\ind_{A_{n,q,\vvt,\vv0}}\mmid\calX,q}\\
& = \esp\pp{1-\sum_{\vv\delta\in\Lambda_d}
\pp{1-\exp\pp{\frac{i\aa{\vv\theta,\vv\delta}\calX}{a_nq^{1/\alpha'}}}}\ind_{A_{n,q,\vvt,\vv\delta}}\mmid\calX,q}.
\end{align*}
where in the second step above we used the fact that $\sum_{\vv\delta\in\{0,1\}^d}\ind_{A_{n,q,\vvt,\vv\delta}} = 1$.
So we arrive at
\[
  \esp\exp\pp{i\sum_{j=1}^{d}\theta_j\frac{S_{\floor{nt_j}}}{a_n}}
 = 1-\sum_{\vv\delta\in\Lambda_d}\esp\pp{\pp{1-\phi_\calX\pp{\frac{\aa{\vv\theta,\vv\delta}}{a_nq^{1/\alpha'}}} }\proba_q(A_{n,q,\vvt,\vv\delta})}.
\]
The key is now to establish
\begin{equation}\label{eq:key est}
\Phi_n:= \esp\bb{\pp{1-\phi_\calX\pp{\frac{\saa{\vv{\theta},\vv{\delta}}}{a_nq^{1/\alpha'}}}}\proba_q(A_{n,q,\vvt,\vv\delta})}
  \sim \frac{ 1}{m_n} \abs{\saa{\vv{\theta},\vv{\delta}}}^\alpha  \mathfrak m_{\alpha,\beta}(\vvt,\vv\delta) \text { as } n\to\infty.
\end{equation}
It then follows that
\begin{align*}
\limn  \esp\exp\pp{i\sum_{j=1}^{d}\theta_j\frac{\widehat{S}_{n}(t_j)}{a_n}}
  &= \limn\bb{
  \esp\exp\pp{i\sum_{j=1}^{d}\theta_j\frac{S_{\floor{nt_j}}}{a_n}}}^{m_n}\\
  & =  \exp\pp{-\sum_{\vv{\delta}\in\Lambda_d}\abs{\saa{\vv{\theta},\vv{\delta}}}^\alpha  \mathfrak m_{\alpha,\beta}(\vvt,\vv\delta)} = \esp\exp\pp{i\sum_{j=1}^{d}
  \theta_j\zeta_{\alpha,\beta}(t_j)}.
\end{align*}

Therefore, it remains to prove \eqref{eq:key est}.
As a preparation, note that for $q>0$ fixed, by standard Poisson approximation for binomial distribution with parameter $(n,q/n)$,  the point process
\[
\summ i1n \delta_{i/n}\inddd{\eta\topp{q/n}_i = 1}
\]
converges in distribution to a Poisson point process on $(0,1)$ with intensity $q$. As a consequence,
\[
\limn\proba_{q}(A_{n,q/n,\vvt,\vv\delta}) = \proba_q\pp{\vv N(q\vv t) = \vv\delta\ \rm mod~2}, \mfa q>0.
\]
Moreover, the above convergence is uniform on any neighborhood of zero, that is,
\[
\limn \sup_{q\in[0,\epsilon]}\frac{\proba_{q}(A_{n,q/n,\vvt,\vv\delta})}{\proba_q\pp{\vv N(q\vv t) = \vv\delta\ {\rm mod}~2}} = 1, \mfa \epsilon>0.
\]
This can be checked by writing explicitly the expressions for the two probabilities. For the sake of simplicity we write only for $d=2, t_1<t_2$ and $\delta_1 = \delta_2 = 1$:
\begin{align*}
\proba_{q}(A_{n,q/n, (t_1,t_2),(1,1)}) &= \frac12\pp{1-\pp{1-\frac {2q}n}^{\floor{nt_1}}}\cdot \frac12\pp{1+\pp{1-\frac {2q}n}^{\floor{nt_2}-\floor{nt_1}}}\\
& \to \frac12 \pp{1-e^{-2qt_1}}\cdot \frac12\pp{1+e^{-2q(t_2-t_1)}} = \proba_q(N(qt_1)\ {\rm odd}, N(qt_2)\ {\rm odd}),
\end{align*}
and the uniform convergence is readily checked.

We first establish a lower bound for $\Phi_n$ in \eqref{eq:key est}.
If we restrict expectation to $q\in[\epsilon/n,\epsilon\inv/n]$ for $\epsilon\in(0,1)$ instead of $q\in[0,1]$, then it follows that
\begin{align}
\Phi_n & \ge \Phi_{n,\epsilon}:=  \int_{\epsilon}^{\epsilon\inv} \pp{1-\phi_\calX\pp{\frac{\saa{\vv{\theta},\vv{\delta}}}{a_nn^{-1/\alpha'}q^{1/\alpha'}}}}\proba_{q}(A_{n,q/n,\vvt,\vv\delta})\pp{\frac qn}^{-\rho}L(n/q)\frac{dq}n\nonumber\\
& \sim  \int_\epsilon^{\epsilon\inv}\sigma_\calX^\alpha|\saa{\vv\theta,\vv\delta}|^\alpha (a_n n^{-1/\alpha'}q^{1/\alpha'})^{-\alpha}\proba_q(\vv N(q\vv t) =\vv\delta \ {\rm mod~2})\pp{\frac qn}^{-
\rho} L(n/q)\frac{dq}n\label{eq:suggestion}\\
& = \frac{|\saa{\vv\theta,\vv\delta}|^\alpha}{m_n}\frac{\beta }{\Gamma(1-\beta)\mathsf C_\alpha}\int_\epsilon^{\epsilon\inv} \frac{L(n/q)}{L(n)}q^{-\gamma-\rho}\proba_q(\vv N(q\vv t) = \vv\delta\ {\rm mod}~2)dq\nonumber\\
& \sim  \frac{|\saa{\vv\theta,\vv\delta}|^\alpha}{m_n}\frac{\beta }{\Gamma(1-\beta)\mathsf C_\alpha}\int_\epsilon^{\epsilon\inv} q^{-\gamma-\rho}\proba_q(\vv N(q\vv t) = \vv\delta\ {\rm mod}~2)dq,\nonumber
\end{align}
which is the same as the right-hand side of \eqref{eq:key est} by taking $\epsilon\downarrow 0$. In the second line above we need $a_nn^{-1/\alpha'}\to\infty$ as $n\to\infty$, which is the same as our assumption on $m_n$ in \eqref{eq:m_n}. Note that   we need to restrict to a compact interval $[\epsilon,\epsilon\inv]$ bounded away from 0 to have the uniform convergence 
in
 the second and the fourth steps above.

It remains to show that $\lim_{\delta\downarrow 0}\limsupn m_n(\Phi_n-\Phi_{n,\epsilon}) = 0$, and we need to work with the intervals $[0,\epsilon/n]$ and $[\epsilon\inv/n,1]$.  This part can be done, and a similar treatment shows up in the proof of Theorem \ref{thm:PPP} (more precisely, see \eqref{eq:Psi} and \eqref{eq:upper 1}). Since Theorem \ref{thm:PPP} is a stronger result than Theorem \ref{thm:1}, we provide full details therein and omit the rest of the proof here.
\section{Limit theorems for point processes}\label{sec:PP}
We first prove Theorem \ref{thm:PPP} 
in Section \ref{sec:proof PPP}. It leads to a second proof of Theorem \ref{thm:1} in Section \ref{sec:second proof}, and also the convergence of the so-called Karlin random sup-measures introduced in \citep{durieu18family} in Section \ref{sec:RSM}.
\subsection{A preparation}
We start by proving a weaker version of Theorem \ref{thm:PPP}.
Recall that
\[
a_n = \pp{C_\calX \frac{\Gamma(1-\beta)}\beta \cdot n^\beta m_nL(n)}^{1/\alpha}.
\]
We shall provide two versions, the second with the alternating signs $(-1)^{\tau_{i,j}\topp {q_i}}$ taken into account but the first not. For each $\ell\in\N$, let $U_{\ell,1:Q_{\beta,\ell}}<\cdots<U_{\ell,Q_{\beta,\ell}:Q_{\beta,\ell}}$ denote the order statistics of $\{U_{\ell,j}\}_{j=1,\dots,Q_{\beta,\ell}}$.
\begin{Prop}\label{prop:PPP}
Under the assumption of Theorem \ref{thm:PPP}, we have
\equh\label{eq:PP_convergence1}
\wt \xi_n:=\summ i1{m_n}\summ j1n\eta_{i,j}\topp{q_i}\ddelta{\calX_i/(a_nq_i^{1/\alpha'}),j/n} \weakto \xi,
\eque
where $\xi$ is as in \eqref{eq:xi}, and
\equh\label{eq:PPP'}
\what\xi_n:=\summ i1{m_n}\summ j1n \eta_{i,j}\topp{q_i}\ddelta{(-1)^{\tau_{i,j}\topp{q_i}}\calX_i/(a_n q_i^{1/\alpha'}),j/n} \weakto \what \xi:=\sif\ell1\summ j1{Q_{\beta,\ell}}\ddelta{(-1)^j \varepsilon_{\ell,j}\Gamma_\ell^{-1/\alpha},U_{j:Q_{\beta,\ell}}}.
\eque
\end{Prop}

Throughout, 
we let $\wt V_\alpha$ denote a symmetrized $\alpha$-Pareto random variable (symmetric and $\proba(|\wt V_\alpha|>x) = x^{-\alpha}, x\ge 1$).
The following is the key step.
\begin{Prop}\label{prop:1}
Introduce
\[
\Omega_n(x):=\ccbb{\frac{|\calX|}{a_nq^{1/\alpha'}}>x, \tau_n\topp q\ne 0}.
\]
Then,
\equh\label{eq:Omega_n}
\proba\pp{\Omega_n(x)} \sim \frac1{m_n} x^{-\alpha}, \mfa x>0,
\eque
and
\equh\label{eq:joint limit}
\calL\pp{\tau_n\topp q, nq,\frac{\calX}{a_nq^{1/\alpha'}}\mmid \Omega_n(x)}
\leadsto \calL\pp{Q_\beta,G(Q_\beta-\beta), x\wt V_\alpha}, \mfa x>0,\eque
where on the right-hand side $G(Q_\beta-\beta)$ is a Gamma random variable with random parameter $Q_\beta-\beta$,  and $\wt V_\alpha$ a symmetrized $\alpha$-Pareto random variable, independent from the first two.
\end{Prop}
The convergence \eqref{eq:joint limit} reads as the weak convergence of the conditional law of the random vector $(\tau_n\topp q,nq, \calX/(a_nq^{1/\alpha'}))$ given $\Omega_n(x)$  to the law of the random vector $(Q_\beta,G(Q_\beta-\beta),x\wt V_\alpha)$.
\begin{Rem}
The convergence to $G(Q_\beta-\beta)$ is not needed in our proof. Nevertheless, it has a probability density in closed form that can be derived as follows. Notice that $G(Q_\beta-\beta)\eqd G(1-\beta)+\summ j1{Q_\beta-1}G_j(1)$, where $G(1-\beta)$ is Gamma with parameter $1-\beta$, $\{G_j(1)\}_{j\in\N}$ are standard exponential random variables and all these random variables and $Q_\beta$ are independent. Recall also the identity that $\esp z^{Q_\beta} = 1-(1-z)^\beta$. Then,  it follows that
\begin{align*}
\esp e^{-\theta G(Q_\beta-\beta)} &  = \esp e^{-\theta G(1-\beta)} \esp \pp{\pp{\esp e^{-\theta G(1)}}^{Q_\beta-1}}\\
& = (1+\theta)^{\beta-1}\esp \pp{(1+\theta)^{-Q_\beta}}(1+\theta) = (1+\theta)^\beta-\theta^\beta.
\end{align*} This is the Laplace transform of the probability density function
\[
\frac{(1-e^{-x})\beta x^{-\beta-1}}{\Gamma(1-\beta)}, \quad x\ge0.
\]
(See \citep[2.2.4.2]{prudnikov92integrals}.)

\end{Rem}

\begin{proof}[Proof of Proposition \ref{prop:1}]
Write
\begin{align}
\proba(\Omega_n(x)) &= \proba\pp{\tau_n\topp q\ne 0,\frac{|\calX|}{a_nq^{1/\alpha'}}>x}
 = \esp \esp \pp{\pp{1-(1-q)^n} \inddd{|\calX|>a_nq^{1/\alpha'}x}\mmid \calX,q}\nonumber\\
& = \esp  \pp{\pp{1-(1-q)^n} \wb F_{|\calX|}\pp{a_n q^{1/\alpha'}x}} = \Psi_{n,\delta}(x)+\Psi_{n,\delta,1}(x)+\Psi_{n,\delta,2}(x),\label{eq:Psi}
\end{align}
with, for $\delta\in(0,1)$,
\begin{align*}
\Psi_{n,\delta}(x) &:= \esp  \pp{\pp{1-(1-q)^n} \wb F_{|\calX|}\pp{a_n q^{1/\alpha'}x};q\in[\delta/n,\delta\inv/n]},\\
\Psi_{n,\delta,1}(x) &:= \esp  \pp{\pp{1-(1-q)^n} \wb F_{|\calX|}\pp{a_n q^{1/\alpha'}x};q>\delta\inv/n},\\
\Psi_{n,\delta,2}(x) &:= \esp  \pp{\pp{1-(1-q)^n} \wb F_{|\calX|}\pp{a_n q^{1/\alpha'}x};q<\delta/n}.\\
\end{align*}
First, for all  $\delta\in(0,1)$,
\begin{align*}
\Psi_{n,\delta}(x)& = \int_{\delta}^{\delta\inv} \pp{\frac qn}^{-\rho}L(n/q)\pp{1-\pp{1-\frac qn }^n}\wb F_{|\calX|}\pp{a_n \pp{\frac qn}^{1/\alpha'}x}\frac{dq}n\\
& \sim C_\calX a_n^{-\alpha}x^{-\alpha} n^{\rho+\gamma-1}L(n) \int_\delta^{\delta\inv}
q^{-\rho}\frac{L(n/q)}{L(n)}\pp{1-\pp{1-\frac qn}^n}q^{-\gamma} dq\\
& \sim \frac\beta{\Gamma(1-\beta)}\frac{ x^{-\alpha}}{m_n}\int_\delta^{\delta\inv}(1-e^{-q})q^{-\gamma-\rho}dq.
\end{align*}
In the second line, we applied
$\wb F_{|\calX|}(a_n(q/n)^{1/\alpha'}x)\sim C_\calX (a_n x)^{-\alpha}(q/n)^{-\gamma}$ uniformly in $q\in[\delta,\delta\inv]$, and for this purpose we shall need $a_n n^{-1/\alpha'}\to \infty$, or equivalently $n^{\beta-\gamma}m_nL(n)\to\infty$, which is our standing condition \eqref{eq:m_n}. Restricted to the same interval we also have $L(n/q)/L(n)\to 1$ uniformly in $q$.  Moreover, we
used the fact that $\limn (1-q/n)^n = e^{-q}$ uniformly over any compact interval $[0,C]$, $C>0$. Recalling that $\beta = \gamma+\rho-1$, we see that
\[
\lim_{\delta\downarrow0}\int_\delta^{\delta\inv}(1-e^{-q})q^{-\gamma-\rho}dq = \int_0^\infty (1-e^{-q})q^{-\beta-1}dq = \frac{\Gamma(1-\beta)}{\beta}.
\]
Thus $\liminfn\proba(\Omega_n(x))\ge \limn \Psi_{n,\delta}(x)$ and this lower bound is tight as it becomes the desired limit in \eqref{eq:Omega_n} as $\delta\downarrow 0$.

 For the upper bound, it remains to show that
 \equh\label{eq:upper 1}
\lim_{\delta\downarrow0}\limsupn m_n\Psi_{n,\delta,i}(x) = 0, \quad i=1,2.
 \eque
 We first show \eqref{eq:upper 1} with $i=1$. This time we use
\[
\Psi_{n,\delta,1}(x)\le  \int_{\delta\inv/n}^1\wb F_{|\calX|}\pp{a_n q^{1/\alpha'}x} q^{-\rho}L(1/q)dq.
\]
 For $n$ large enough, $\wb F_{|\calX|}(a_n q^{1/\alpha'}x)\le (1+\delta) C_\calX(a_n q^{1/\alpha'}x)^{-\alpha}$, for all $q$ in the range of the integral (by uniform convergence of regularly varying function). Therefore, the above is bounded by, for $n$ large enough,
\begin{align*}
(1+\delta)C_\calX\int_{\delta\inv/n}^1 (a_nq^{1/\alpha'}x)^{-\alpha} q^{-\rho}L(1/q)dq &= \frac{(1+\delta)C_\calX}{(a_nx)^{\alpha}}\int_1^{\delta n}q^{\gamma+\rho-2}L(q)dq  \\
& \sim \frac{(1+\delta)C_\calX}{(a_nx)^{\alpha}}
\beta^{-1}
 (\delta n)^\beta L(\delta n)\sim \frac1{m_n}\frac {1+\delta}{\Gamma(1-\beta)}\delta^\beta x^{-\alpha}.
\end{align*}
In the second step we invoked Karamata's theorem. Now \eqref{eq:upper 1} with $i=1$ follows.

Next we show \eqref{eq:upper 1} with $i=2$.
Pick $D_\delta$ be such that for all $x>D_\delta$,
 \equh\label{eq:D_delta}
 \sup_{x>D_\delta}
\frac{\wb F_{|\calX|}(x)}{C_\calX x^{-\alpha}}<1+\delta.
 \eque Write $D_{n,\delta,x} = (D_\delta/(a_n x))^{\alpha'}$. One checks readily that the convergence $\limn nD_{n,\delta,x} = 0$ is the same as $\limn n^{\beta-\gamma}m_n L(n) = \infty$.
 We decompose further the integral area with respect to $q\in [0,\delta/n]$ into $[0,D_{n,\delta,x}]$ and $[
D_{n, \delta, x}
 ,\delta/n]$, and write respectively $\Psi_{n,\delta,2}(x) = \Psi_{n,\delta,2,1}(x)+\Psi_{n,\delta,2,2}(x)$.
 Then,
 \begin{align*}
\Psi_{n,\delta,2,2}(x) & = \int_{nD_{n,\delta,x}}^{\delta} \pp{\frac qn}^{-\rho}L(n/q)\pp{1-\pp{1-\frac qn }^n}\wb F_{|\calX|}\pp{a_n \pp{\frac qn}^{1/\alpha'}x}\frac{dq}n\nonumber\\
& \le (1+\delta) {C_\calX a_n^{-\alpha}x^{-\alpha}}\int_{nD_{n,\delta,x}}^{\delta} \pp{\frac qn}^{-\rho}L(n/q)\pp{1-\pp{1-\frac qn}^n}\pp{\frac qn}^{-\gamma} \frac{dq}n\nonumber\\
& \le (1+\delta)C_\calX x^{-\alpha} \frac{n^{
\beta
}L(n)}{a_n^\alpha}\int_0^\delta\frac{L(n/q)}{L(n)}\pp{1-\pp{1-\frac qn}^n}q^{-\gamma-\rho}dq.
\end{align*}
In the first inequality above we applied \eqref{eq:D_delta}, and the second we used $\beta+1 = \rho+\gamma$ and extend the lower bound of the integral region to zero.
Then, for $n$ large enough, so that $L(n/q)/L(n)<(1+\delta)q^{-\delta}$ for all $q\in(0,\delta)$ (Potter's bound \citep{resnick07heavy}), the above is bounded by, for $\delta\in(0,2-(\gamma+\rho))$,
\[
\frac1{m_n}\frac{ (1+\delta)^2 \beta x^{-\alpha}}{\Gamma(1-\beta)}\int_0^\delta\pp{1-\pp{1-\frac qn}^n}q^{-\gamma-\rho-\delta}dq \sim
\frac1{m_n}\frac{ (1+\delta)^2 \beta x^{-\alpha}}{\Gamma(1-\beta)}\int_0^\delta(1-e^{-q})q^{-\gamma-\rho-\delta}dq.
\]
The right-hand side above has the expression $R_\delta(x)/m_n$ with $\lim_{\delta\downarrow 0}R_\delta(x) = 0$. Next,  for $\delta>0$ small enough,
\begin{align*}
\Psi_{n,\delta,2,1} & \le \int_0^{nD_{n,\delta,x}}\pp{1-\pp{1-\frac qn}^n}\pp{\frac qn}^{-\rho}L(n/q)\frac{dq}n\\
& = L(n) n^{\rho-1}\int_0^{nD_{n,\delta,x}} \pp{1-\pp{1-\frac qn}^n}\frac{L(n/q)}{L(n)}q^{-\rho}dq \\
& \le (1+\delta) L(n) n^{\rho-1}\int_0^{nD_{n,\delta,x}} (1-e^{-q})q^{-\rho-\delta}dq  \sim \frac{1+\delta}{2-\rho-\delta}L(n) n^{\rho-1} \pp{n \pp{\frac {D_\delta}{a_n x}}^{\alpha'}}^{2-\rho-\delta},
\end{align*}
where in the second inequality above we used Potter's bound again, for $n$ large enough.
We want to show the above is asymptotically of a smaller order than $m_n\inv$, or equivalently, dropping the dependence on $\rho,
D_\delta
,x$,
\[
(m_nL(n))^{1-(2-\rho-\delta)/\gamma}n^{1-\delta-\beta(2-\rho-\delta)/\gamma} =  (m_nL(n))^{(\gamma+\rho+\delta-2)/\gamma}n^{(\gamma(1-\delta)-\beta(2-\rho-\delta))/\gamma}  \to0.
\]
Indeed, since $\rho+\gamma\in(1,2)$ and $\rho<1$, one could take  $\delta>0$ small enough (precisely, $\rho+\delta<1, \gamma+\rho+\delta<2$) so that
\[
n^{\frac{\gamma(1-\delta)-\beta(2-\rho-\delta)}{\gamma+\delta-(2-\rho)} }m_nL(n)  = n^{\frac{\beta(2-\rho-\delta)-\gamma(1-\delta)}{2-(\gamma+\rho+\delta)}}m_nL(n) \ge n^{\beta-\gamma}m_nL(n)\to\infty,
\]
where the last step is our standing assumption. Combining the above we have proved \eqref{eq:upper 1} with $i=2$, and hence \eqref{eq:Omega_n}.

Similarly,
one can show that
\begin{align*}
\proba  \pp{nq\in(a,b), \tau_n\topp q = k, \frac{|\calX|}{a_nq^{1/\alpha'}}>x}
&= \int_{a/n}^{b/n}\binom nk q^k(1-q)^{n-k}\wb F_{|\calX|}\pp{a_n q^{1/\alpha'}x}q^{-\rho}L(1/q)dq\\
& \sim\frac1{m_n} \frac{\beta x^{-\alpha}}{\Gamma(1-\beta)}\int_a^b \frac{n^k}{k!}\pp{\frac qn}^{k-\rho}\pp{1-\frac qn}^n(n^{\beta-\gamma}q^\gamma)\inv \frac{dq}n \\
&\sim \frac1{m_n} \frac{\beta x^{-\alpha}}{\Gamma(1-\beta)}\int_a^b\frac{e^{-q}}{k!}q^{k-1-\beta}dq.
\end{align*}
The asymptotic equivalence above follows from the dominated convergence theorem and is much simpler than before. We omit the details.
So, we have
\begin{align*}
\proba\pp{nq\in(a,b), \tau_n\topp q = k, \frac{|\calX|}{a_nq^{1/\alpha'}}>x} & \sim \frac{x^{-\alpha}}{m_n} \frac\beta{\Gamma(1-\beta)}\frac{\Gamma(k-\beta)}{\Gamma(k+1)} \proba(G(k-\beta)\in(a,b))\\ & = \frac{x^{-\alpha}}{m_n}\proba(Q_\beta = k, G(k-\beta)\in(a,b)),
\end{align*}
where $G(k-\beta)$ is a Gamma random variable with parameter $k-\beta$, independent from $Q_\beta$. The desired \eqref{eq:joint limit} then follows.
\end{proof}

\begin{proof}[Proof of Proposition \ref{prop:PPP}]
The second convergence \eqref{eq:PPP'} can be proved  in exactly the same way as \eqref{eq:PP_convergence1}, 
 and the only difference is the alternating signs in both the discrete-time aggregated model and the limit point process. Therefore, we prove only \eqref{eq:PP_convergence1} for the sake of notational simplicity.

We prove by computing the Laplace transform. Let $f(x,y)$ be a bounded and continuous
 function from $\R\times[0,1]$ to $\R_+$ such that $f(x,y) = 0$ for all $|x|\le \kappa$ for some $\kappa>0$.  Then,
\begin{align*}
\esp e^{-\wt\xi_n(f)}& = \esp\exp\pp{-\summ i1{m_n}\summ j1n f\pp{\calX_i/(a_nq_i^{1/\alpha'}),j/n}\eta_{i,j}\topp{q_i}}\\
& = \pp{\esp\exp\pp{-\summ j1nf\pp{\calX/(a_n q^{1/\alpha'}),j/n}\eta_j\topp q}}^{m_n} = \pp{\proba(\Omega_n(\kappa)) \Psi_n(\kappa) + 1-\proba(\Omega_n(\kappa))}^{m_n},
\end{align*}
with
\[
\Psi_n(\kappa):=\esp\pp{\exp\pp{-\summ j1nf\pp{\calX/(a_n q^{1/\alpha'}),j/n}\eta_j\topp q}\mmid\Omega_n(\kappa)}.
\]
Then, by Proposition \ref{prop:1}, writing $\Omega_{n,\ell}(\kappa) := \{|\calX|/(a_n q^{1/\alpha'})>\kappa, \tau_n\topp q = \ell\}$,
\begin{align*}
\Psi_n(\kappa)&=\sif\ell1 \esp\pp{\exp\pp{-\summ j1nf\pp{\calX/(a_n q^{1/\alpha'}),j/n}\eta_j\topp q}\mmid \Omega_{n,\ell}(\kappa)}\proba\pp{\Omega_{n,\ell}(\kappa)\mmid\Omega_n(\kappa)}\\
& \to \sif\ell1 \esp\exp\pp{-\summ j1\ell f(\kappa \wt V_\alpha,U_j)}\proba(Q_\beta = \ell) = \esp\exp\pp{-\summ j1{Q_\beta}f(\kappa \wt V_\alpha,U_j)}=:\Psi(\kappa).
\end{align*}
The convergence above follows from the observation that given $\Omega_{n,\ell}(\kappa)$, $\{\eta_{j}\topp{q}\}_{j=1,\dots,n}$ is exchangeable; from this we derive that since $\summ j1n \eta_j\topp q = \ell$, the law of $\{j/n\}_{j=1,\dots,n, \eta_j\topp q = 1}$ follows the law of $\ell$-sampling without replacement from $\{1,\dots,n\}$, which has the limit as $\{U_j\}_{j=1,\dots,\ell}$, and their independence from $\wt V_\alpha$ follows from the conditional independence of $\{\eta_j\topp q\}_{j=1,\dots,n}$ from $\calX/(a_n q^{1/\alpha'})$.
Therefore, it follows that, recalling $\proba(\Omega_n(\kappa))\sim \kappa^{-\alpha}/m_n$ in \eqref{eq:Omega_n},
\begin{align*}
 \esp e^{-\wt\xi_n(f)} & = \pp{1-\proba(\Omega_n(\kappa))(1-\Psi_n(\kappa))}^{m_n} \\
 & \to \exp\pp{-\limn m_n\proba(\Omega_n(\kappa))(1-\Psi(\kappa))} = \exp\pp{- \kappa^{-\alpha}(1-\Psi(\kappa))}.
\end{align*}
At the same time, let $N_\kappa$ denote a Poisson random variable with intensity $\kappa^{-\alpha}$. Then,
\[
\esp e^{-\xi(f)}  = \esp\exp\pp{-\summ i1{N_\kappa}\summ j1{Q_{\beta,i}}f(\wt V_{\alpha,i},U_{i,j})} = \esp\pp{\Psi(\kappa)^{N_\kappa}} = e^{-\kappa^{-\alpha}(1-\Psi(\kappa))},
\]
where $(\wt V_{\alpha,i},Q_{\beta,i},\{U_{i,j}\}_{j\in\N}), i\in\N$ are i.i.d.~copies of $(\wt V_\alpha, Q_\beta,\{U_j\}_{j\in\N})$.
This completes the proof.
\end{proof}

\subsection{Proof of Theorem \ref{thm:PPP}}\label{sec:proof PPP}
Set
\[
\wt X_{n,j} :=\frac1{a_n}\summ i1{m_n}\frac{\calX_i\eta_{i,j}\topp{q_i}}{q_i^{1/\alpha'}}.
\]
Recall the notations around \eqref{eq:notations}. Recall also that $\rho+\gamma = \beta+1\in(1,2)$, $\gamma = \alpha/\alpha'>0$ and $\rho<1$.
Introduce  for $\epsilon>0$,
\[
\wt X_{n,j,\epsilon} : = \frac1{a_n}\summ i1{m_n}\frac{\calX_i\eta_{i,j}\topp{q_i}}{q_i^{1/\alpha'}}\inddd{|\calX_i|>a_nq_i^{1/\alpha'}\epsilon}.
\]
The idea of the proof is to compare
\equh\label{eq:xi_n_epsilon}
\xi_{n,\epsilon}:=\summ j1n \ddelta{\wt X_{n,j,\epsilon},j/n} \qmand \wt \xi_{n,\epsilon}:=\sum_{\substack{i=1,\dots,m_n\\ |\calX_i|>a_nq_i^{1/\alpha'}\epsilon}}\summ j1n \eta_{i,j}\topp{q_i}\ddelta{\calX_i/(a_nq_i^{1/\alpha'}),j/n}.
\eque
We have seen in Proposition
\ref{prop:PPP}
that the latter above converges to the desired point process $\xi$ in
\eqref{eq:xi}
 restricted to
$([-\infty,-\epsilon]\cup[\epsilon,\infty])\times[0,1]$.
 Introduce also
\[
\what C_{n,\epsilon}(i) := \ccbb{j=1,\dots,n: \frac{|\calX_i|}{q_i^{1/\alpha'}}\eta_{i,j}\topp{q_i}>a_n\epsilon} \qmand \what C_{n,\epsilon} :=\bigcup_{i=1}^{m_n} \what C_{n,\epsilon}(i),
\]
and furthermore
\equh\label{eq:epsilon_n}
\epsilon_n := n^{-\beta_0/\alpha}, n\in\N,
\eque
for any $\beta_0\in(0,\beta)
$. We begin by analyzing $\wt X_{n,j} - \wt X_{n,j,\epsilon}$, which is the same as $Z_{n,j,\epsilon,\epsilon_n}+W_{n,j,\epsilon_n}$ with
\begin{align*}
Z_{n,j,\epsilon,\epsilon_n}&:= \summ i1{m_n}\frac{\calX_i}{a_nq_i^{1/\alpha'}}\eta_{i,j}\topp{q_i}\inddd{|\calX_i|/ \pp{a_n q_i^{1/\alpha'}}\in[\epsilon_n,\epsilon]},\\
W_{n,j,\epsilon_n}&:= \summ i1{m_n}\frac{\calX_i}{a_nq_i^{1/\alpha'}}\eta_{i,j}\topp{q_i}\inddd{|\calX_i|/\pp{ a_n q_i^{1/\alpha'}}<\epsilon_n}, \quad j=1,\dots,n.
\end{align*}
\begin{Lem}\label{lem:Z}
We have, for $r  \equiv r_{\beta,\beta_0}:= \floor{1/(\beta-\beta_0)}+1$,
\equh\label{eq:Z}
\limn \proba\pp{\max_{j=1,\dots,n}|Z_{n,j,\epsilon,\epsilon_n}|\ge r\epsilon} = 0, \mfa \epsilon>0,
\eque
and
\equh\label{eq:Z1}
\limn\proba\pp{\max_{j\in\what C_{n,\epsilon}}|Z_{n,j,\epsilon,\epsilon_n}|
>
0} = 0.
\eque
\end{Lem}
\begin{Lem}\label{lem:W}
We have
\equh\label{eq:W}
\limn \proba\pp{\max_{j=1,\dots,n}|W_{n,j,\epsilon_n}|>\lambda} = 0, \mfa \lambda>0.
\eque
\end{Lem}
\begin{proof}[Proof of Lemma \ref{lem:Z}]
We shall need
\equh\label{eq:exceedance}
\proba\pp{\frac{|\calX|}{a_n q^{1/\alpha'}}>\epsilon_n, \eta_1\topp q=1} \le C (a_n\epsilon_n)^{-\alpha}, \mfa n\in\N.
\eque
Here and below, we let $C$ denote a positive constant that may change from line to line.
To see the above,
we write the probability on the left-hand side of \eqref{eq:exceedance} as $\int_0^1q^{1-\rho}L(1/q)\wb F_{|\calX|}(a_n q^{1/\alpha'}\epsilon_n)dq$, 
and let $d_n$ be such that
\equh\label{eq:d_n}
d_n\downarrow 0, \quad
d_n(a_n\epsilon_n)^{\alpha'}\to\infty \qmand  d_n^{2-\rho}L(1/d_n)(a_n\epsilon_n)^{\alpha}\to 0.
\eque
(One readily checks that such a sequence exists since
$(a_n\epsilon_n)^{-\alpha'}\ll (a_n\epsilon_n)^{-\alpha/(2-\rho)}$,
which is equivalent to $\alpha'>\alpha/(2-\rho)$, or $2-\rho-\gamma>0$.)
  Decompose the integral into $\int_0^{d_n}$ and $\int_{d_n}^1$, we bound the first by
$\int_0^{d_n} q^{1-\rho}L(1/q)dq \sim (2-\rho)\inv d_n^{2-\rho}L(1/d_n)$,
and the second by
\[
C\int_{d_n}^1 q^{1-\rho}L(1/q) (a_n \epsilon_n q^{1/\alpha'})^{-\alpha}dq 
=
 C(a_n\epsilon_n)^{-\alpha}\int_{d_n}^1 q^{1-\rho-\gamma}L(1/q)dq\sim C(a_n\epsilon_n)^{-\alpha}.
\]
Note that in the above, we need
$d_n(a_n\epsilon_n)^{\alpha'}\to\infty$ (the second condition in \eqref{eq:d_n}), and the third condition in \eqref{eq:d_n} now implies that the integral over $[d_n,1]$ is dominant.
We have proved \eqref{eq:exceedance}.

Now, to prove \eqref{eq:Z}, it suffices to prove
\[
\limn \proba\pp{\max_{j=1,\dots,n}\summ i1{m_n}\eta_{i,j}\topp{q_i}\inddd{|\calX_i|>a_nq_i^{1/\alpha'}\epsilon_n}\ge r} = 0.
\]
In words, with probability going to zero, at some location $j$ there are more than  $r$ different indices $i$ such that $|\calX_i|$ is large and also $\eta_{i,j}\topp{q_i} = 1$ (in the complement of this event, the largest possible value of 
$|Z_{n,j,\epsilon,\epsilon_n}|$
 is
$(r-1)\epsilon_n$, for all $j$).  An upper bound of the probability of interest above is then
\[
n\binom{m_n}r
\left(\proba\pp{\frac{|\calX|}{a_n q^{1/\alpha'}}>\epsilon_n, \eta_1\topp q=1}\right)^r
 \le C nm_n^r (a_n\epsilon_n)^{-\alpha r}.
\]
We see that our choices of $\beta_0\in(0,\beta)$ and $r$ entail that the right-hand side above decays to zero.  We have thus proved \eqref{eq:Z}.

Next, we prove \eqref{eq:Z1}. By a similar argument as above, we have
\[
\proba\pp{\max_{j\in\what C_{n,\epsilon}}\summ i1{m_n}\eta_{i,j}\topp{q_i}\inddd{|\calX_i|>a_nq_i^{1/\alpha'}\epsilon_n}>0}
\le \proba\pp{|\what C_{n,\epsilon}|>K} + K m_n\proba\pp{\frac{|\calX|}{a_n q^{1/\alpha'}}>\epsilon_n, \eta_1\topp q=1}.
\]
The second term on the right-hand side above is bounded from above by $CKm_n(a_n\epsilon_n)^{-\alpha}\to 0$,
for all $K>0$ fixed. So we have
\[
\limsupn\proba\pp{\max_{j\in\what C_{n,\epsilon}}\summ i1{m_n}\eta_{i,j}\topp{q_i}\inddd{|\calX_i|>a_nq_i^{1/\alpha'}\epsilon_n}>0}\le \limsupn \proba\pp{|\what C_{n,\epsilon}|>K},
\]
where the right-hand side tends to zero by taking $K\to\infty$.
Indeed,
first notice that
by Proposition \ref{prop:PPP},
$|\what C_{n,\epsilon}| \le \summ i1{m_n}|\what C_{n,\epsilon}(i)|\weakto \summ i1{N_\epsilon}Q_{\beta,i}$,
where $N_\epsilon$ is a Poisson random variable with parameter $\epsilon^{-\alpha}$, and $\{Q_{\beta,i}\}_{i\in\N}$ are i.i.d.~random variables independent from $N_\epsilon$.  Therefore,
\[
\limsup_{n\to\infty}
 \proba\pp{|\what C_{n,\epsilon}|>K }  \le \proba\pp{\summ i1{N_\epsilon} Q_{\beta,i}>K}.
\]
(This
 inequality is actually an equality, as later on we shall see that $|\what C_{n,\epsilon}| = \summ i1{m_n}|\what C_{n,\epsilon}(i)|$ with probability tending to one; i.e., $\limn\proba(\calE_{n,1,\epsilon}^c) = 0$ in the proof of Theorem \ref{thm:PPP}.)
It thus follows that
\equh\label{eq:what C_n}
\lim_{K\to\infty}\limn \proba\pp{|\what C_{n,\epsilon}|>K }  = 0, \mfa \epsilon>0.
\eque
We have proved \eqref{eq:Z1}.
\end{proof}

\begin{proof}[Proof of Lemma \ref{lem:W}]
Now we prove 
\eqref{eq:W}.
Write
\[
W_{n,1,\epsilon_n} = \summ i1{m_n}V_{n,i,\epsilon_n} \qmwith V_{n,i,\epsilon_n}:=\frac{\calX_i}{a_n q_i^{1/\alpha'}}\eta_{i,1}\topp{q_i}\inddd{|\calX_i|/\pp{a_n q^{1/\alpha'}_i}<\epsilon_n}.
\]
Observe that $|V_{n,i,\epsilon_n}|\le \epsilon_n$ and write $w_n := m_n\esp V_{n,1,\epsilon_n}^2$.
By union bound first and then the Bernstein inequality  \citep[(2.10)]{boucheron13concentration} , we have
\equh\label{eq:Bernstein}
\proba\pp{\max_{j=1,\dots,n}|W_{n,j,\epsilon_n}|>\lambda}\le n \proba(|W_{n,1,\epsilon_n}|>\lambda)
 \le 2n\exp\pp{-\frac{\lambda^2}{2(w_n + \epsilon_n \lambda/3)}},
\eque
for all $\lambda>0$ and $n\in\N$.
We shall compute at the end
\equh\label{eq:w_n}
w_n\le \begin{cases}
\displaystyle C \frac{m_n}{a_n^\alpha}\epsilon_n^{2-\alpha}, & \mbox{ if } \alpha\in(0,2),\\\\
\displaystyle C \frac{m_n}{a_n^2}(1+\log(a_n\epsilon_n)_+), & \mbox{ if } \alpha=2,\\\\
\displaystyle C \frac{m_n}{a_n^2}, & \mbox{ if } \alpha>2.
\end{cases}
\eque
Then, by \eqref{eq:Bernstein} and our choice of $\epsilon_n$, it suffices to check that $w_n\to 0$ at a polynomial rate. This is true for $\alpha\in(0,2)$, and for $\alpha\ge 2$ an additional assumption on $m_n$ is needed.
Indeed, with $\alpha=2$ the $\log (a_n\epsilon_n)$ might be problematic if $m_n$ grows at an exponential rate, while any polynomial growth of $m_n$ would cause no problem; and with $\alpha>2$,
\[
\frac{m_n}{a_n^2}  = C\frac{a_n^{\alpha-2}}{n^\beta L(n)}
 = C\frac{m_n^{1-2/\alpha}}{(n^\beta L(n))^{2/\alpha}}\to 0
\]
at a polynomial rate is guaranteed by $m_n\le Cn^\kappa$ for any $\kappa<2\beta/(\alpha-2)$.
Therefore, the desired \eqref{eq:W} follows from \eqref{eq:Bernstein} and  \eqref{eq:w_n}.

It remains to prove \eqref{eq:w_n}. We have
\begin{align}
w_n&= m_n \esp\pp{\pp{\frac\calX{a_n q^{1/\alpha'}}}^2\eta\topp q\inddd{|\calX|<a_n\epsilon_n q^{1/\alpha'}}}\nonumber\\
& = \frac{m_n}{a_n^2} \int_0^1 q^{1-\rho-2/\alpha'}L(1/q) \esp \pp{\calX^2\inddd{|\calX|<a_n\epsilon_n q^{1/\alpha'}}}dq.\label{eq:w_n integral}
\end{align}
Now the discussions shall depend on the values of $\alpha>0$ in three cases.

(i) Assume $\alpha<2$. Introduce a parameter $d_n = (a_n\epsilon_n)^{-\alpha'}\downarrow 0$ (we no longer need the same constraints on $d_n$ as in \eqref{eq:d_n}
 before as we only need an upper bound now). Again decompose the integral in \eqref{eq:w_n integral} into two parts 
on $\int_0^{d_n}$ and $\int_{d_n}^1$, respectively.
 Applying Karamata's theorem on the expectation,
the second part (with the factor $m_n/a_n^2$ in front) can be bounded by,
\[
C\frac{m_n}{a_n^2}\int_{d_n}^1 q^{1-\rho-2/\alpha'}L(1/q) (a_n\epsilon_nq^{1/\alpha'})^{2-\alpha}dq
= C\frac{m_n \epsilon_n^{2-\alpha}}{a_n^\alpha}\int_{d_n}^1 q^{1-\rho-\gamma}L(1/q)dq.
\]
The first part can be bounded by
\begin{align}
m_n\epsilon_n^2\int_0^{d_n}q^{1-\rho}L(1/q) dq & \le Cm_n\epsilon_n^2d_n^{2-\rho}L(1/d_n)\nonumber\\
&=C\frac{m_n\epsilon_n^{2-\alpha}}{a_n^\alpha}(a_n\epsilon_n)^{\alpha-\alpha'(2-\rho)}L(1/d_n)
= o\pp{\frac{m_n\epsilon_n^{2-\alpha}}{a_n^\alpha}}.\label{eq:0,d_n}
\end{align}
(Note that $\alpha-\alpha'(2-\rho) = \alpha'(\gamma+\rho-2)<0$.)

(ii)
 If $\alpha>2$, then
\[
w_n\le \esp\calX^2 \frac{m_n}{a_n^2}\int_0^1q^{1-\rho-2/\alpha'}L(1/q)dq \le C \frac{m_n}{a_n^2}.
\]

(iii) If $\alpha=2$, under the assumption $\proba(|\calX|>x)\sim C_\calX x^{-2}$, there exists a constant $C$ 
 such that
\equh\label{eq:gamma_0}
\esp\pp{|\calX|^2\inddd{|\calX|<x}}\le 1+C(\log x)_+, \mfa x>0.
\eque
Then, \eqref{eq:w_n integral} with the integrals restricted to $[0,d_n]$ (we use the same bound as in \eqref{eq:0,d_n}) and $[d_n,1]$ (we use the bound \eqref{eq:gamma_0} above) are bounded from above by respectively
\[
C\frac{m_n}{a_n^2}(a_n\epsilon_n)^{2-\alpha'(2-\rho)}L(1/d_n) \qmand C\frac{m_n}{a_n^2}(1+(\log(a_n\epsilon_n))_+).
\]
Again the part over $[d_n,1]$ is dominant.
We have thus proved \eqref{eq:w_n}.
\end{proof}
\begin{proof}[Proof of Theorem \ref{thm:PPP}]
Consider a Lipschitz continuous and bounded
non-negative
 function $f(x,y)$ such that $f(x,y) = 0$ for all $x\in[-\kappa,\kappa]$, with Lipschitz constant $C_f$.
Let $r = r_{\beta,\beta_0} = \floor{1/(\beta-\beta_0)}+1$ as in Lemma \ref{lem:Z}, and $\epsilon\in(0,\kappa/(r+1))$.  Introduce
\begin{align*}
\calE_{n,1,\epsilon}&:=\ccbb{\ccbb{\what C_{n,
\epsilon}
(i)
}_{i=1,\dots,m_n} \mbox{ are all disjoint }},\\
\calE_{n,2,\epsilon}&:=\ccbb{\max_{j=1,\dots,n}\abs{\wt X_{n,j}-\wt X_{n,j,\epsilon}}\le (r+1)\epsilon},\\
\calE_{n,3,\epsilon,K}&:=\ccbb{|\what C_{n,\epsilon}|\le K},\\
\calE_{n,4,\epsilon,\lambda}&:=\ccbb{\max_{j\in\what C_{n,\epsilon}}\abs{\wt X_{n,j}-\wt X_{n,j,\epsilon}}\le \lambda},
\end{align*}
and $\calE_{n,\epsilon,K,\lambda} := \calE_{n,1,\epsilon}\cap \calE_{n,2,\epsilon} \cap \calE_{n,3,\epsilon,K}\cap \calE_{n,4,\epsilon,\lambda}$. Recall $\xi_{n,\epsilon}$ and $\wt \xi_{n,\epsilon}$ in \eqref{eq:xi_n_epsilon}.
The key relation in the approximation  is for all $K>0$,
\equh\label{eq:sandwich}
  e^{-\wt\xi_{n,\epsilon}(f)}e^{-\lambda KC_f} \le e^{-\xi_{n}(f)}  \le  e^{-\wt\xi_{n,\epsilon}(f)}e^{\lambda KC_f}, \mbox{ restricted to $\calE_{n,\epsilon,K,\lambda}$}.
\eque
We prove the upper-bound part only as  the lower-bound part is similar.
Restricted to $\calE_{n,\epsilon,K,\lambda}$, we have
\begin{align*}
e^{-\xi_n(f)} & = \exp\pp{-\sum_{j\in\what C_{n,\epsilon}}f\pp{\wt X_{n,j},j/n}} \le \exp\pp{-\sum_{j\in\what C_{n,\epsilon}}f\pp{\wt X_{n,j,\epsilon},j/n}}e^{\lambda KC_f} \\
&= e^{-\xi_{n,\epsilon}(f)}e^{\lambda KC_f} =e^{-\wt \xi_{n,\epsilon}(f)}e^{\lambda KC_f},
\end{align*}
where we used the restrictions to $\calE_{n,2,\epsilon}$ in the first equality (since for $j\notin\what C_{n,\epsilon}$, $\wt X_{n,j} = \wt X_{n,j}-\wt X_{n,j,\epsilon}$, which is small when restricted to $\calE_{n,2,\epsilon}$), to $\calE_{n,3,\epsilon,K}\cap\calE_{n,4,\epsilon,\lambda}$ in the first inequality by Lipschitz continuity,  and to $\calE_{n,1,\epsilon}$ in the third equality, respectively. (In the third equality, we used the observation that  restricted to the event $\calE_{n,1,\epsilon}$, $\xi_{n,\epsilon} = \wt\xi_{n,\epsilon}$.
Indeed, on the event $\calE_{n,1,\epsilon}$ if $\wt X_{n,j,\epsilon}\ne 0$ for some $j$, then necessarily $\wt X_{n,j,\epsilon} = \calX_i\eta_{i,j}\topp{q_i}/(a_nq_i^{1/\alpha'})$ for a unique index $i \in\{1,\dots,m_n\}$ and for all other $i'$, $| \calX_{i'}|\eta_{i',j}\topp{q_{i'}}
\leq
a_nq_{i'}^{1/\alpha'}\epsilon$.)

Recall our choice of $\epsilon_n$ in \eqref{eq:epsilon_n}. Then, the upper bound in \eqref{eq:sandwich} becomes
\begin{align*}
\limsupn \esp e^{-\xi_n(f)} &
\leq \limsupn \esp \pp{e^{-\wt\xi_{n,\epsilon}(f)+\lambda KC_f}\ind_{\calE_{n,\epsilon,K,\lambda}}} + \limsupn \esp\pp{e^{-\xi_n(f)}\ind_{\calE_{n,\epsilon,K,\lambda}^c}}\\
& \le  \limsupn \esp \pp{e^{-\wt\xi_{n,\epsilon}(f)+\lambda KC_f}}
 + \limsupn \proba\pp{\calE_{n,\epsilon,K,\lambda}^c}\\
& =
\esp e^{-\xi(f)}\cdot e^{\lambda KC_f}+ \limsupn \proba\pp{\calE_{n,\epsilon,K,\lambda}^c}.
\end{align*}
In the last step we used 
first  $\wt\xi_{n,\epsilon}(f) = \wt\xi_n(f)$ thanks to the assumption that $f(x,y) = 0$ for $x\in[-\kappa,\kappa]$, and then $\limn\esp e^{-\wt \xi_n(f)} = \esp e^{-\xi (f)}$ by Proposition 
\ref{prop:PPP}.
A similar argument yields the lower bound
\begin{align*}
\liminfn \esp e^{-\xi_n(f)} & \geq \liminfn \esp \pp{e^{-\wt\xi_{n,\epsilon}(f)-\lambda KC_f}\ind_{\calE_{n,\epsilon,K,\lambda}}}\\
&\geq \liminfn \esp \pp{e^{-\wt\xi_{n,\epsilon}(f)-\lambda KC_f}}-\limsupn \proba\pp{\calE_{n,\epsilon,K,\lambda}^c}\\
&=\esp e^{-\xi(f)}\cdot e^{-\lambda KC_f}- \limsupn \proba\pp{\calE_{n,\epsilon,K,\lambda}^c}.
\end{align*}
Combining 
these two bounds gives
\begin{align*}
\esp e^{-\xi(f)}e^{-\lambda KC_f}-\limsupn\proba\pp{\calE_{n,\epsilon,K,\lambda}^c}& \le  \liminfn \esp e^{-\xi_n(f)}\\
&\le\limsupn \esp e^{-\xi_n(f)}\le \esp e^{-\xi(f)}e^{\lambda KC_f}+\limsupn \proba\pp{\calE_{n,\epsilon,K,\lambda}^c}.
\end{align*}
Now, the desired convergence $\limn \esp e^{-\xi_n(f)} = \esp e^{-\xi(f)}$  follows by first taking $\lambda\downarrow 0$ and then $K\to\infty$, combined with the following facts:
\[
\limsupn\proba\pp{\calE_{n,\epsilon,K,\lambda}^c} \le  \limsupn \proba\pp{\calE_{n,3,\epsilon,K}^c}, \mfa K,\lambda>0,
\]
and
\[
\lim_{K\to\infty}\limsupn \proba\pp{\calE_{n,3,\epsilon,K}^c} = 0.
\]
To see the above we examine  each of the four events separately.
\begin{enumerate}[(i)]
\item $\limn\proba(\calE_{n,1,\epsilon}^c)=0$. Asymptotically, there are $N_\epsilon$ (a Poisson random variable with mean $\epsilon^{-\alpha}$) number  of $\what C_{n,\epsilon}(i)$ that are non-empty. Therefore, it suffices to show that
\equh\label{eq:two_intersection}
\limn \proba\pp{\what C_{n,\epsilon}(1)\cap \what C_{n,\epsilon}(2) \ne \emptyset\mmid \what C_{n,\epsilon}(i)\ne\emptyset, i=1,2}= 0.
\eque
Again, we can restrict to the event $|\what C_{n,\epsilon}(i)| \le K_0, i=1,2$ for $K_0\in\N$, and it is clear that
\[
\limn\proba\pp{\what C_{n,\epsilon}(1)\cap \what C_{n,\epsilon}(2)\ne\emptyset,  |\what C_{n,\epsilon}(i)| \le K_0, i=1,2} = 0,
\]
and by (\ref{eq:what C_n})
\[
\lim_{
K_0
\to\infty}\limsupn\proba\pp{|\what C_{n,\epsilon}(i)|>K_0, \mbox{ for $i=1$ or $2$}} = 0.
\]
The desired \eqref{eq:two_intersection} then follows.

\item $\limn\proba(\calE_{n,2,\epsilon}^c)=0$. This follows from  \eqref{eq:Z} and \eqref{eq:W}, and the identity that  $\wt X_{n,j} - \wt X_{n,j,\epsilon} = Z_{n,j,\epsilon,\epsilon_n}+W_{n,j,\epsilon_n}$.

\item $\lim_{K\to\infty}\limn\proba(\calE_{n,3,\epsilon,K}^c) = 0$. We already proved this in \eqref{eq:what C_n}.

\item $\limn\proba(\calE_{n,4,\epsilon,\lambda}^c) = 0$. To see this,  use the relation
\[
\proba\pp{\calE_{n,4,\epsilon,\lambda}^c}\le \proba\pp{ \max_{j\in\what C_{n,\epsilon}}|Z_{n,j,\epsilon,\epsilon_n}| > 0}+\proba \pp{\max_{j=1,\dots,n}|W_{n,j,\epsilon_n}|>\lambda}.
\]
Then recall \eqref{eq:Z1} and \eqref{eq:W}.
\end{enumerate}
We have completed the proof.
\end{proof}
\subsection{A second proof of Theorem \ref{thm:1}}\label{sec:second proof}
In the case of i.i.d.~random variables with regularly-varying tails of tail index $\alpha\in(0,2)$, it is a classical result that once the point-process convergence
is
 established, the functional central limit theorem holds \citep[proof of Proposition 3.4]{resnick86point}. Here we can also obtain another proof of Theorem \ref{thm:1} following Proposition \ref{prop:1}. However, as mentioned in Remark \ref{rem:tightness}, the tightness is hard for Karlin stable processes. We only manage to prove the convergence of finite-dimensional distributions.

The proof consists of an approximation argument.
Let $T_{2,\epsilon}$ be as in \citep[proof of Proposition 3.4]{resnick86point}. This is a mapping from
$\mathfrak M_p(\wb\R\setminus\{0\}\times[0,1])$ to $D([0,1])$, with, for any $\zeta  = \sum_i\ddelta{y_i,u_i}\in\mathfrak M_p(\wb \R\setminus\{0\}\times[0,1])$,
\[
[T_{2,\epsilon}\zeta](t):= \sum_iy_i\inddd{u_i\le t, |y_i|>\epsilon}, t\in[0,1].
\]
Thus, by continuous mapping theorem applied to Proposition \ref{prop:PPP},
$T_{2,\epsilon}\what\xi_{n} \weakto T_{2,\epsilon}\what\xi$ ($T_{2,\epsilon}$ is almost surely continuous with respect to law induced by $\what\xi$),  which is the same as (compare with \eqref{eq:zeta series})
\[
\ccbb{\frac1{a_n}\summ i1{m_n}\frac{\calX_i}{q_i^{1/\alpha'}}\inddd{\tau_{i,\floor{nt}}\topp{q_i}\ \rm odd}\inddd{|\calX_i|>a_nq_i^{1/\alpha'}\epsilon}}_{t\in[0,1]}
\weakto \left\{\sif\ell1 \frac{\varepsilon_\ell}{\Gamma_\ell^{1/\alpha}}\inddd{\summ j1{Q_{\beta,\ell}}\inddd{U_{\ell,j}\le t}\ \rm odd}\inddd{\Gamma_\ell^{-1/\alpha}>\epsilon}\right\}_{t\in[0,1]}
\]
in $D([0,1])$.
The above implies the convergence of finite-dimensional distribution of the truncated process, and it remains to show that for every $t\in[0,1]$,
\[
\lim_{\epsilon\downarrow0}\limsupn\proba\pp{\abs{\frac1{a_n}\summ i1{m_n}\frac{\calX_i}{q_i^{1/\alpha'}}\inddd{|\calX_i|\le a_nq_i^{1/\alpha'}\epsilon}\inddd{\tau_{i,\floor{nt}}\topp{q_i}\ \rm odd}}>\lambda} = 0, \mfa \lambda>0.
\]
(See \citep[Theorem 2]{dehling09new}.)
It suffices to prove for a fixed $t$, and without loss of generality we take $t=1$. In this case the above follows from Chebychev inequality and, for all $\epsilon>0$,
\equh\label{eq:v_n}
\limsupn v_{n,\epsilon}\le C\epsilon ^{2-\alpha}\qmwith
v_{n,\epsilon}:=m_n\esp\pp{\pp{\frac{\calX}{a_nq^{1/\alpha'}}}^2\inddd{|\calX|\le a_nq^{1/\alpha'}\epsilon}\inddd{\tau_{n}\topp{q}\ \rm odd}}.
\eque
We first compute $v_{n,\epsilon}$, with the expectation restricted to $q\in[1/n,1]$. An upper bound is then (bounding the second indicator function by 1), for $n$ large enough,
\[
\frac{m_n}{a_n^2}\int_{1/n}^1 q^{-\rho-2/\alpha'} L(1/q)\esp_q \pp{\calX^2\inddd{|\calX|\le a_n q^{1/\alpha'}\epsilon}}dq \le \frac{Cm_n\epsilon^{2-\alpha}}{a_n^\alpha}\int_{1/n}^1 q^{-\rho-\gamma}L(1/q) dq \le  C\epsilon^{2-\alpha}.
\]
(More precisely, $\epsilon>0$ is fixed, $C$ can be taken independent of $\epsilon$, while the above holds only for all $n>n_{C,\epsilon}$ for some $n_{C,\epsilon}$.) For $v_{n,\epsilon}$ with the expectation restricted to $q\in[0,1/n]$, note that then
$\proba_q\pp{\tau\topp q_n\ \rm odd}  = (1-(1-2q)^n)/2$ and
\[
\sup_{q\in[0,1/n]}\frac{(1-(1-2q)^n)}{qn}= 2.
\]
Therefore,
\begin{align*}
\frac{m_n}{a_n^2}&\int_0^{1/n}\esp_q\pp{\calX^2\inddd{|\calX|\le a_nq^{1/\alpha'}\epsilon}}\proba_q\pp{\tau_{\floor{nt}}\topp{q}\ {\rm odd}}q^{-\rho-2/\alpha'}L(1/q)dq\\
&\le \frac{m_n}{a_n^2}\int_0^{1/n}\esp_q\pp{\calX^2\inddd{|\calX|\le a_nq^{1/\alpha'}\epsilon}}n q^{1-\rho-2/\alpha'}L(1/q)dq\\
&\le \frac{C m_n n}{a_n^2}\int_0^{1/n} \left(a_nq^{1/\alpha'}\epsilon\right)^{2-\alpha}q^{1-\rho-2/\alpha'}L(1/q)dq
= \frac{Cm_n n}{a_n^\alpha}\epsilon^{2-\alpha}\int_n^\infty q^{\beta-2}L(q)dq\\
&\le \frac{Cm_n n}{a_n^\alpha}n^{\beta-1}L(n)\epsilon^{2-\alpha}=C\epsilon^{2-\alpha}.
\end{align*}
We have thus proved \eqref{eq:v_n}.
\begin{Rem}
If we want to enhance the result to a functional central limit theorem in $D([0,1])$, a sufficient condition would be
\[
\lim_{\epsilon\downarrow0}\limsupn\proba\pp{\sup_{t\in[0,1]}\abs{\frac1{a_n}\summ i1{m_n}\frac{\calX_i}{q_i^{1/\alpha'}}\inddd{|\calX_i|\le a_nq_i^{1/\alpha'}\epsilon}\inddd{\tau_{i,\floor{nt}}\topp{q_i}\ \rm odd}}>\lambda} = 0, \mfa \lambda>0.
\]
Whether the above is true remains an open question.
\end{Rem}

\subsection{A limit theorem for Karlin random sup-measures}\label{sec:RSM}
Now we explain how Theorem \ref{thm:PPP} entails the convergence of random sup-measures. Random sup-measures provide a natural framework to characterize scaling limits of extremes, although they are not commonly used yet in the literature. For background of random sup-measures, see \citep{obrien90stationary,vervaat97random,molchanov17theory}.
 For the sake of simplicity, we shall  treat random sup-measures as $\alpha$-Fr\'echet max-stable set-indexed process $\{\calM_{\alpha,\beta}(I)\}_{I\in\calI}$, with $\calI$ the collection of all open sets of $[0,1]$, denoted by
\[
\mab(I):=\sup_{\ell\in\N}\frac1{\Gamma_\ell^{1/\alpha}}\inddd{\pp{\bigcup_{j=1}^{Q_{\beta,\ell}}\{U_{\ell,j}\}}\cap I\ne\emptyset}, I\subset \calI,
\]
and prove the convergence of finite-dimensional distributions of the set-indexed processes (for max-stable processes, see \citep{dehaan84spectral,kabluchko09spectral,stoev10max}).
For $\calM_{\alpha,\beta}$, it has the following multivariate $\alpha$-Fr\'echet finite-dimensional distributions (although we do not need to work with the explicit formula):
\[
\proba\pp{\calM_{\alpha,\beta}(I_1)\le x_1,\dots,\calM_{\alpha,\beta}(I_d)\le x_d}
= \exp\pp{-\esp\pp{\max_{k=1,\dots,d}\frac{\inddd{\calC_{Q_\beta}\cap I_k\ne\emptyset}}{
x_k^\alpha
}}},
\]
for all $I_1,\dots,I_d\in\calI, x_1,\dots,x_d>0$.

The following result on the convergence of max-stable processes can be strengthened immediately to convergence of random sup-measures (which is defined for {\em all} subsets of $[0,1]$). We just mention that Karlin random sup-measures are translation-invariant and $\beta/\alpha$-self-similar, and they are a special case of the recently introduced Choquet random sup-measures \citep{molchanov16max}. We refer to \citep{durieu18family} for more results on the Karlin random sup-measures.

Introduce
\[
M_n(I):=\max_{j/n\in I}\frac1{a_n}\abs{\summ i1{m_n}\frac{\calX_i}{q^{1/\alpha'}_i}\eta_{i,j}\topp{q_j}}, \quad I\subset \calI, n\in\N,
\]

\begin{Coro}
Under the assumption of Theorem \ref{thm:PPP},
\[
\ccbb{M_n(I)}_{I\in\calI}\fddto \ccbb{\calM_{\alpha,\beta}(I)}_{I\in\calI}.
\]
\end{Coro}
\begin{proof}
By definition, it suffices to 
show
\equh\label{eq:RSM fdd convergence}
\limn\proba\pp{M_n(I_1)\le x_1,\dots,M_n(I_d)\le x_d} = \proba\pp{\calM_{\alpha,\beta}(I_1)\le x_1,\dots,\calM_{\alpha,\beta}(I_d)\le x_d},
\eque
for all $d\in\N,x_i>0, I_i\in\calI, i=1,\dots,d$. Now, Theorem \ref{thm:PPP} implies that, ignoring the signs of the values and working with point processes in $\mathfrak M_p((0,\infty]\times[0,1])$,
\[\xi_n^*:=\summ j1n\ddelta{\abs{\summ i1{m_n}\calX_i\eta\topp{q_i}_{i,j}/(a_nq_i^{1/\alpha'})},j/n} \weakto \xi^*:=\sif\ell1\summ j1{Q_{\beta,j}}\ddelta{\Gamma_\ell^{-1/\alpha},U_\ell}.
\]
The above then implies in particular, with $B := \bigcup_{k=1}^d\pp{(x_k,\infty]\times I_k}$,
\[
\limn\proba\pp{\xi_n^*(B)= 0} = \proba(\xi^*(B) = 0).
\]
The above is exactly the desired convergence in \eqref{eq:RSM fdd convergence}. This completes the proof.
\end{proof}
\subsection*{Acknowledgements} YW's research was partially supported by Army Research Office, US (W911NF-20-1-0139)
, YS's research was supported by NSERC, Canada (2020-04356).
 \bibliographystyle{apalike}
\bibliography{references,references18}
\end{document}